\documentclass[11pt,a4paper]{amsart}
\usepackage{amsmath,amssymb}
\usepackage[latin1]{inputenc}
\usepackage{verbatim}

\newcommand\err{\mathbb R}

\newcommand{\eps}{\varepsilon}

\newcommand\Xs{X^\ast}

\newcommand\xs{x^\ast}

\newcommand\ys{y^\ast}

\newcommand{\HB}{\text{H{\kern -0.35em}B}}

\DeclareMathOperator{\spann}{span}

\DeclareMathOperator{\sign}{\mbox{sign}}
\newcounter{abcd}
{\setcounter{abcd}{0}
\begin{list}%
{{\rm (\alph{abcd})}} 
{\usecounter{abcd}
\parsep=\parskip
\topsep=1pt plus 2pt minus 1pt
\itemsep=1pt plus 2pt minus 1pt
\leftmargin=3\baselineskip \labelsep=.6\baselineskip
\labelwidth=2.4\baselineskip

\rightmargin 0pt}%
}%
{\end{list}}

\title{An improvement of a theorem of Heinrich, Mankiewicz, Sims,
  and Yost}

\author[T.~A.~Abrahamsen]{Trond A. Abrahamsen}
\address{Department of Mathematics, University of Agder, Postbox 422\\
4604 Kristiansand, Norway.} \email{trond.a.abrahamsen@uia.no}
\urladdr{http://home.uia.no/trondaa/index.php3}


\newtheorem{thm}{Theorem}[section]
\newtheorem{prop}[thm]{Proposition}
\newtheorem{lem}[thm]{Lemma}

\theoremstyle{definition}

\newtheorem{defn}[thm]{Definition}

\theoremstyle{remark}

\newtheorem{rem}[thm]{Remark}

\begin{document}

\begin{abstract}
  Heinrich, Mankiewicz, Sims, and Yost proved that every separable
  subspace of a Banach space $Y$ is contained in a separable ideal in
  $Y$. We improve this result by replacing the term ``ideal'' with the term
  ``almost isometric ideal''. As a consequence of this we obtain, in terms of
  subspaces, characterizations of diameter 2
  properties, the Daugavet property along with the properties of being
  an almost square space and an octahedral space.

\end{abstract}

\maketitle

\section{Introduction}\label{sec1}

Let $Y$ be a Banach space and $X$ a subspace of $Y$. Recall that
$X$ is an \emph{ideal} in $Y$ if $X^\perp$, the annihilator of $X$, is
the kernel of a contractive projection on the dual $Y^*$ of $Y$. A linear operator
$\varphi$ from $X^*$ to $Y^*$ is called
a \emph{Hahn-Banach extension operator} if
$\varphi(x^*)(x) = x^*(x)$ and $\|\varphi(x^*)\|=\|x^*\|$
for all $x \in X$ and $x^* \in X^*$. We denote by $\HB(X,Y)$ the set
of all Hahn-Banach extension operators from $X^*$ to $Y^*$. We say
that $X$ is locally 1-complemented in $Y$ if for
every $\eps > 0$ and every finite dimensional subspace $E$ of $Y$
there exists a linear operator $T: E \to X$ such that $Te = e$ for all
$e \in E \cap X$ and $\|T\| \le 1 + \eps$. The fact that a Banach space
is locally 1-complemented in its bidual is commonly referred to as
\emph{the Principle of Local Reflexivity (PLR)}.  

The following theorem is a collection of known results.

\begin{thm}\label{thm:opfk}
  Let $X$ be a subspace of a Banach space $Y$.
  The following statements are equivalent.
  \begin{itemize}
  \item[(a)] $X$ is an ideal in $Y$. 
  \item[(b)] There exists $\varphi \in \HB(X,Y)$.
  \item[(c)] $Y$ is locally 1-complemented in $X$.
  \item[(d)] There exists $\varphi \in \HB(X,Y)$ such that for every
    $\eps > 0$, every finite dimensional subspace $E$ of $Y$ and
    every finite dimensional subspace $F$ of $X^*$ there exists a
    linear operator $T: E \to X$ such that
    \begin{itemize}
      \item[(d1)] $Te = e$ for all $e \in E \cap X$, 
      \item[(d2)]$\|Te\| \le (1 + \eps)\|e\|$ for every $e \in
        E$, and
      \item[(d3)] $\varphi f(e) = f(Te)$ for every $e \in E$, $f \in F$.
    \end{itemize}
  \end{itemize}
\end{thm}

The equivalence of (a), (b), and (c) were independently discovered by
Fakhoury \cite{Fak4} and Kalton \cite{Kal}. Later Oja and P{\~o}ldvere
\cite{MR2262909} showed that these in turn are
equivalent to statement (d). 

The following result is essentially due to Heinrich and Mankiewicz
\cite{MR675426}. Sims and Yost, however, gave in \cite{SiYo2} another
proof of this result, using a finite dimensional lemma and a
compactness argument due to Lindenstrauss \cite{Lin}. As stated below
the result appears for the first time in \cite[III.Lemma~4.3]{HWW}. 

\begin{thm}\label{thm:sims-yost-hww}
 Let $Y$ be a Banach space, $X$ a separable subspace of $Y$, and $W$ a
 separable subspace of $Y^*$. Then there
 exists a separable subspace $Z$ of $Y$ containing $X$ and $\varphi \in
 \HB(Z,Y)$ such that $\varphi (Z^*) \supset W$.      
\end{thm}

In the language of ideals this result says that every separable subspace of a
Banach space $Y$ is contained in a separable ideal in $Y$. Thus every
non-separable Banach space contains an infinite number of
ideals. Looked upon in this way ideals seems to occur quite
frequently. 

The following stronger form of an ideal was introduced and
studied in \cite{ALN2}.

 \begin{defn}\label{def:ai-hbext}
  A subspace $X$ of a Banach space $Y$ is said to be an \emph{almost
    isometric ideal (ai-ideal)} if for every $\eps >
  0$ and every finite dimensional subspace $E$ of $Y$, there exists a
  linear operator $T:E \to X$ which satisfies (d1) in
  Theorem \ref{thm:opfk} as well as 
  \begin{itemize}
    \item [(d2')] $(1-\eps)\|e\|\leq\|Te\|\leq (1+\eps)\|e\|$ for
    $e\in E$.
  \end{itemize}
\end{defn}

  In \cite{ALN2} the following was shown.

\begin{thm}\label{thm:aln2}
  Let $X$ be a subspace of a Banach space $Y$.
  The following statements are equivalent.
  \begin{itemize}
    \item[(a)]$X$ is an ai-ideal in $Y$.
    \item[(b)]There exists $\varphi \in \HB(X,Y)$ such that for every
      $\eps > 0$, every finite dimensional
    subspace $E$ of $Y$,  and every finite
  dimensional subspace $F$ of $X^*$ there exists a
    linear operator $T: E \to X$ which satisfies (d1) and (d3) in
    Theorem \ref{thm:opfk} and (d2') in Definition \ref{def:ai-hbext}.      
  \end{itemize}
\end{thm} 
  
  The $\varphi$ in Theorem \ref{thm:aln2} is called an
  \emph{almost isometric} Hahn-Banach extension operator associated
  with the ai-ideal $X$ in $Y$. We denote by $\HB_{ai}(X,Y)$ the set
  of such operators.

The main result of this paper is an improvement of Theorem
\ref{thm:sims-yost-hww} in which the
Hahn-Banach extension operator is replaced by an almost isometric one.

\begin{thm}[Main Theorem]\label{thm:abr}
 Let $Y$ be a Banach space, $X$ a separable subspace of $Y$, and $W$ a
 separable subspace of $Y^*$. Then there exists a separable subspace
 $Z$ of $Y$ containing $X$ and $\varphi \in \HB_{ai}(Z,Y)$ such that
 $\varphi (Z^*) \supset W$.      
\end{thm}

 So every separable subspace of a Banach space is contained in a
 separable ai-ideal. Thus ai-ideals seems to occur just as
 frequently as ideals. Nevertheless, being an ai-ideal is strictly stronger
 than being an ideal. This was proved in \cite[Example~1]{ALN2} but
 can e.g. also be seen from Theorem
 \ref{thm:lindenstrauss-gurariy} and the two paragraphs that follow. Theorem
 \ref{thm:lindenstrauss-gurariy} is a collection of
 \cite[Proposition~3.4]{Fak4} and \cite[Theorem~4.3]{ALN2}.

\begin{thm}\label{thm:lindenstrauss-gurariy} For a Banach space $X$ the
  following statements are equivalent:
\begin{itemize}
  \item [(i)] $X$ is a Lindenstrauss (resp. Gurari{\u\i}) space.
  \item [(ii)] $X$ is an ideal (resp. ai-ideal) in every superspace.
\end{itemize}
\end{thm}

Recall that a \emph{Lindenstrauss space} is a Banach space with a dual
isometric to $L_1(\mu)$ for some positive measure $\mu$. A Banach
space $X$ is called a \emph{Gurari{\u\i} space} if it has the property that
whenever $\eps>0$, $E$ is a finite-dimensional Banach space, $T_E:E\to
X$ is isometric and $F$ is a finite-dimensional Banach space with $E
\subset F$, then there exists a linear operator $T_F:F\to X$ such that
\begin{itemize}
  \item [(i)] $T_F(f)=T_E(f)$ for all $f\in E$, and
  \item [(ii)] $(1-\eps)\|f\|\leq\|T_F f\|\leq (1+\eps)\|f\|$ for
    all $f\in F$.
\end{itemize}

From Theorem \ref{thm:lindenstrauss-gurariy} we see that the class of
Gurari{\u\i} spaces is a subclass of the class of Lindenstrauss
spaces. In \cite{MR0200697} Gurari{\u\i} showed that this subclass is non-empty as he constructed the first separable
Gurari{\u\i} space. Later Lusky \cite{MR0433198} proved that all
separable Gurari{\u\i} spaces are in fact
linearly isometric. Also non-separable Gurari{\u\i} spaces exist (see
e.g. \cite{MR2977626}). But, no
Gurari{\u\i} space is a dual space as e.g. the unit ball of such a
space contains no extreme points
\cite[Proposition~3.3]{ALLN}. The bidual of a Lindenstrauss
space is, however, again a Lindenstrauss space \cite{Lin-Mem}. Thus we see
that the classes of separable and non-separable Gurari{\u\i} spaces
are non-empty proper subclasses of respectively the classes of
separable and non-separable Lindenstrauss spaces.   

 Let us now relate the notion of an ai-ideal to the well established
 notion of a strict ideal (see e.g. \cite{GKS}, \cite{MR2559177}, 
 \cite{MR1840956}, and \cite{A2}). We say that $X$ is a \emph{strict ideal} in
 $Y$ if $X$ is an ideal in $Y$ with an associated $\varphi \in
 \HB(X,Y)$ whose range is 1-norming for $Y$, i.e. for every $y \in Y$
 we have $\|y\| = \sup\{y^*(y): y^* \in \varphi(X^*) \cap S_{Y^*}\}$
 where $S_{Y^*}$ is the unit sphere of $Y^*$.
Let $\HB_{s}(X,Y) = \{\varphi \in \HB(X,Y): \varphi \text{ is
   strict}\}.$ Using the PLR it is straightforward to show that every
 strict ideal is an ai-ideal. However, the converse it not true (see
 e.g. \cite[Example~1]{ALN2} and \cite[Remark~3.2]{ALLN}). We can sum
 up the last paragraphs by 
\[\HB(X,Y) \supset \HB_{ai}(X,Y) \supset \HB_{s}(X,Y),\] 
where the containment may be proper.    
     
The paper is organized as follows: In Section \ref{sec2} we give a proof
of Theorem \ref{thm:abr}. In Section \ref{sec3} we use this theorem
to obtain characterizations of diameter 2 properties, the Daugavet
property as well as the properties of being an almost square space and
an octahedral space. 

We will consider real Banach spaces only (though many of the results
are true in the complex case as well).  The notation used
is mostly standard and is, if considered necessary, explained as the text proceeds. 
   
\section{The main theorem}\label{sec2}
The proof of Theorem \ref{thm:abr} depends on Lemma \ref{lem-hww+}
below. The roots of this lemma goes back to \cite{Lin} and \cite[III.Lemma~4.2]{HWW}. Lemma \ref{lem-hww+}  differs from
\cite[III.Lemma~4.2]{HWW} simply by the fact that the partial conclusion
\begin{enumerate}
  \item[ii)] $\|Tx\| \le (1 + \eps)\|x\|$ for every $x \in E$, 
\end{enumerate} 
in \cite[III.Lemma~4.2]{HWW} is replaced by the stronger
partial conclusion ii') in Lemma \ref{lem-hww+}. The proof of Lemma
\ref{lem-hww+} is interestingly enough already contained in the
proof of Lemma \cite[III.Lemma~4.2]{HWW}. This is perhaps not so easy
to spot at first glance. So to make this clearer, we present a
complete proof here. 

\begin{lem}\label{lem-hww+}
  Let $Y$ be a Banach space, $B$ a finite dimensional subspace of $Y$, $k
  \in \mathbb N$, $\eps > 0$, and $C$ a finite subset in $Y^*$. Then
  there is a finite dimensional subspace $Z$ containing $B$ such
  that for every subspace $E$ of $Y$ containing $B$ and satisfying
  $\dim E/B \le k$ one can find a linear operator $T:E \to Z$ such that
  \begin{enumerate}
    \item[i)]$Ty = y$ for every $y \in B$,
    \item[ii')] $(1 - \eps)\|y\| \le \|Ty\| \le (1 + \eps)\|y\|$ for
      every $y \in E$,
    \item[iii)]$|f(Ty) - f(y)| \le \eps \|y\|$ for every $y \in E$ and
      $f \in C$. 
  \end{enumerate}
\end{lem}

\begin{proof}
  Choose $\delta > 0$ such that $\delta < \eps$ and $(1 + \delta)^{-1}
  > 1 - \eps$. Let $C = \{f_1, \ldots, f_m\} \subset Y^*$ and $P$ a
  projection   on $Y$ onto $B$. Put $U = \ker P$, the kernel of $P$. Then we can write $Y  = B \oplus U$. Choose $M$ so large that
  \[M > \frac{5k\|I - P\|}{\delta} \text{ and } \frac{M + 1}{M - 1} < 1
  + \delta.\]
  Let $(b_\rho)_{\rho \le r}$ and $(\lambda_\sigma)_{\sigma \le s}$ be
  finite $1/M$-nets for $\{b \in B: \|b\| \le M\}$ and $S_{\ell_1(k)}$
  respectively ($\ell_1(k)$ denotes the $k$-dimensional $\ell_1$
  space). Let $B_U$ be the unit ball of $U$ and define $\phi:(B_U)^k \to \mathbb R^{rs} \times \mathbb
  R^{mk} = \mathbb R^{rs + mk}$, by   \[\phi(u_1, \ldots, u_k) =
  \left( (\|b_\rho + \sum_{\kappa = 1}^k   \lambda_{\kappa,
      \sigma}u_\kappa\|)_{\rho \le r, \sigma \le s},
    (f_\mu(u_\kappa))_{\mu \le m, \kappa \le k }\right).\] Since
  $\phi(B_U)^k$ is totally bounded, we can find $(\frak u_\nu)_{\nu
    \le n} \subset (B_U)^k$ such that $(\phi \frak u_\nu)_{\nu \le n}$
  is a   finite $1/M$-net for $\phi(B_U)^k$ where
  we may take any norm om
  $\mathbb R^{rs + mk}$ for which the coefficient functionals have
  norm $\le 1$. Put 
  \[Z = B \oplus \spann\{u_{\kappa, \nu}: \kappa \le k, \nu \le n\}.\]
  Now, given a subspace $E \supset B$ with $\dim E/B = k$, there are
  $u_1, \ldots, u_k \in U$ such that $E = B \oplus \spann\{u_\kappa: \kappa \le
  k\}$. By Auerbach's lemma we can choose $\frak u = (u_1, \ldots,
  u_k)$ such that
  \begin{equation}
    \label{eq:3}
    \|u_\kappa\| = 1, 1 \le \kappa \le k \text{ and } \|\sum_{\kappa =
    1}^k \lambda_\kappa u_\kappa\|\ge \frac{1}{k}\sum_{\kappa = 1}^k |\lambda_\kappa|
  \text{ for all } (\lambda_\kappa) \in \mathbb R^k.
  \end{equation}
  Indeed, find $u^*_\kappa \in S_{\spann\{u_\kappa:1, \ldots k\}^*}$ with
  $u^*_{\kappa} (u_j) = 0$ if $\kappa \not=j$ and
  $\sign(\lambda_{\kappa})$ otherwise. Then the norm of
  $u^*=1/k\sum_{\kappa=1}^k u^*_{\kappa}$ is $\le 1$ and we have
  \begin{align*}
    \|\sum_{\kappa = 1}^k\lambda_\kappa u_\kappa\| & \ge
                                                     u^*(\sum_{\kappa
                                                     = 1}^k
                                                     \lambda_\kappa u_\kappa) =
    \frac{1}{k}\sum_{\kappa = 1}^k |\lambda_\kappa|.  
  \end{align*}
     
  This means that there is $\nu \le n$ such that 
  \[\|\phi \frak u - \phi \frak u_\nu\| < \frac{1}{M},\]
  i.e. with $\lambda \frak u = \sum_{\kappa = 1}^k \lambda_\kappa
  u_\kappa$ where $\lambda = (\lambda_1, \ldots, \lambda_k) \in
  \mathbb R^k$ we have 
  \begin{equation}
   \label{eq:4}
    \bigg| \|b_\rho + \lambda_\sigma\frak u\| - \|b_\rho +
    \lambda_\sigma \frak u_\nu\| \bigg| < \frac{1}{M} \text{ for all } \rho
  \le r, \sigma \le s,
  \end{equation}
  and 
  \begin{equation}
    \label{eq:5} 
    \bigg|f_\mu(u_\kappa) - f_\mu(u_{\kappa, \nu})\bigg| <
    \frac{1}{M} \text{ for all } \mu \le m, \kappa \le k.
  \end{equation} 
  Now, define $T: E \to Z$ by 
  \[T(b + \lambda \frak u) = b + \lambda
  \frak u_\nu.\]
  Clearly $T$ is the identity on $B$. To show $(1 -
  \eps)\|y\| \le \|Ty\| \le (1 + \eps)\|y\|$ for all $y \in E$, it
  suffices to prove 
  \begin{equation}
    \label{eq:6}
    (1 - \eps)\|b + \lambda \frak u\| \le \|b +
    \lambda \frak u_\nu\| \le (1 + \eps) \|b + \lambda\frak u\|  \text{ for
    } \|\lambda\|_{\ell_1(k)} = 1.
  \end{equation}
   
  To this end assume first $\|b\| \le M$. Then $\|b - b_\rho\| < \frac{1}{M}$ for
  some $\rho$ and $\|\lambda - \lambda_\sigma\| < \frac{1}{M}$ for
  some $\sigma$. Thus
  \begin{align*}
    \|b + \lambda \frak u_\nu\|& \le \|b_p + \lambda_\sigma \frak u_\nu\| + \|b 
    - b_\sigma\| + \|\lambda \frak u_\nu - \lambda_\sigma \frak u_\nu \|\\
                        & < \|b_p + \lambda_\sigma \frak u_\nu\| +
                        \frac{2}{M}\\
                        &  \overset{(\ref{eq:4})}{<}  \|b_\rho +
                        \lambda_\sigma \frak u\| + \frac{3}{M}\\
                        & \le \|b + \lambda \frak u\| +  \|b_\rho - b\| +
                        \|\lambda_\sigma \frak u - \lambda \frak u \| +
                        \frac{3}{M}\\
                        &< \|b + \lambda \frak u\| + \frac{5}{M}.
  \end{align*}
 Similarly we also get $\|b + \lambda \frak u\| < \|b + \lambda \frak u_\nu\|
 + \frac{5}{M}$, so we have 
 \[\|b + \lambda \frak u\| - \frac{5}{M}< \|b + \lambda \frak u_\nu\| < \|b +
 \lambda \frak u\| + \frac{5}{M}.\]
 Also
 \begin{align*}
   \|b + \lambda \frak u\| &\ge  \frac{1}{\|I - P\|}\|(I - P)(b +
 \lambda \frak u)\|\\
                           & = \frac{1}{\|I - P\|}\| \sum_{\kappa =
                             1}^k \lambda_\kappa u_\kappa\|\\
                           & \overset{(\ref{eq:3})}{\ge} \frac{1}{k\|I
                             - P\|}\sum_{\kappa =
                             1}^k|\lambda_\kappa|\\
                           & = \frac{1}{k\|I
                             - P\|} > \frac{5}{\delta M},
 \end{align*}
 so $\eps \|b + \lambda \frak u\| \ge \frac{5}{M}$ and thus
 (\ref{eq:6}) holds for $\|b\| \le M$.

 For $\|b\| > M$ we have
 \[\|b\| - 1 \le \|b + \lambda \frak u_\nu\| \le \|b\| + 1 \text{ and
 } \|b\| - 1 \le \|b +
 \lambda \frak u\| \le \|b\| + 1,\]
 so both  
 \[\frac{\|b + \lambda \frak u\|}{\|b + \lambda \frak u_\nu\|} \text{
   and } \frac{\|b + \lambda \frak u_\nu\|}{\|b + \lambda \frak u\|}
 \text{ are }\le
 \frac{\|b\| + 1}{\|b\| - 1} < \frac{M + 1}{M - 1} < 1 + \delta,\]
 as $y \to \frac{y+1}{y-1}$ is a positive and decreasing function for
 $y > 1$. Thus (\ref{eq:6}) holds also for $\|b\| > M$. 

 Finally for any $y = b + \sum_{\kappa = 1}^k \lambda_\kappa u_\kappa
 \in E$ we have
 \begin{align*}
   |f_\mu(y) - f_\mu(Ty)| &= \bigg| f_\mu\bigg(\sum_{\kappa =
     1}^k\lambda_\kappa(u_\kappa - u_{\kappa, \nu}) \bigg)\bigg|\\
                         &\le \sum_{\kappa =
     1}^k|\lambda_\kappa||f_\mu(u_\kappa - u_{\kappa, \nu})|\\
                         &
                         \overset{(\ref{eq:5})}{\le}\frac{1}{M}\sum_{\kappa
                           =
     1}^k |\lambda_\kappa|\\
                         &\overset{(\ref{eq:3})}{\le}\frac{k}{M}\|\sum_{\kappa =
     1}^k \lambda_\kappa u_\kappa\|\\
                        & \le  \frac{k\|I-P\|}{M}\|y\| < \frac{\eps}{5}\|y\|. 
 \end{align*}
 \end{proof}

By the proof of \cite[Theorem \ref{thm:aln2}]{ALN2} the
following holds. 

\begin{lem}\label{lem:ai-equi}
  Let $X$ be an ideal in $Y$ and let $\varphi \in \HB(X,Y)$.
  Then the following statements are equivalent.
  \begin{itemize}
    \item[(a)]$\varphi \in \HB_{ai}(X,Y)$. 
    \item[(b)]For every $\delta, \eps > 0$, for every finite dimensional
  subspace $E$ of $Y$, and every finite dimensional subspace $F$ of
  $X^*$ there exists a linear operator $T:E \to X$ which satisfies
  \begin{itemize}
    \item[(d1')] $\|Te - e\| \le \eps\|e\|$ for every $e \in E \cap
      X$,
    \item[(d2')] in Definition \ref{def:ai-hbext}, and
    \item[(d3')] $|\varphi f(e) - f(Te)| < \delta\|e\|\cdot\|f\|$ for
      every $e \in E$, $f \in F$.
  \end{itemize} 
  \end{itemize}
\end{lem}

Along with Lemma \ref{lem-hww+} we will use Lemma \ref{lem:ai-equi}
to prove our Main Theorem.

\begin{proof}[Proof of Theorem \ref{thm:abr}]
  The first part of the proof is identical to that of
  \cite[III.Lemma~4.3]{HWW} except at the crucial point where we use
  Lemma \ref{lem-hww+} in place of \cite[III.Lemma~4.2]{HWW}. The
  reward for this is that we are able to prove the existence of a
  Hahn-Banach extension operator $\varphi$ for which statement (b)
  in Lemma \ref{lem:ai-equi} holds. 
 
  Let $(x_n)$ be a sequence dense in $X$ and $(f_n)$ a sequence dense
  in $W$. Starting with $M_1 =\{0\}$ we inductively define subspaces
  $M_n$ as follows: Put $B_n = \spann(M_n, x_n)$, $C_n = \{f_1,
  \ldots, f_n\}$, and let $M_{n+1}$ be the subspace $Z$ given by Lemma
  \ref{lem-hww+} when $B = B_n, k = n, \eps = \frac{1}{n}$, and $C =
  C_n$. Without loss of generality assume $\dim M_{n+1}/B_n \ge n +
  1$. Clearly $M   = \overline{\cup M_n}$ is separable and contains
  $X$. For $n \in \mathbb N$ define
  \[I_n = \{E \subset Y: B_n \subset E, \dim E/B_n \le n\}\]
  and put \[I = \cup I_n.\]
  Since 
  \begin{equation}
    \label{eq:7}
     E \in I_n, F \in I_m \Rightarrow E \oplus F \oplus B_{\dim E +
    \dim F} \in I_{\dim E + \dim F}, 
  \end{equation}
  we have that $I$ is a directed set. Moreover, it is clear that every
  finite dimensional subspace $F$ of $Y$ is
  contained in some $E \in I$. Just take $E=F \oplus B_{\dim F}$. Then $E
  \in I_{\dim F}$. 
  
  Note that the condition $\dim M_{n+1}/B_n \ge n+1$
  implies $\dim B_{n+1}/B_n \ge n+1$. This easily gives that for each
  $E \in I$ there is a unique $n \in \mathbb N$ such that $E \in
  I_n$. So by Lemma \ref{lem-hww+} there exists a linear operator
  $T_E: E \to M_{n+1} \subset M$ such that $T_E|_{B_n} = I_{B_n}$,
  $(1-\frac{1}{n})\|y\|\le\|T_Ey\|  \le (1 +  \frac{1}{n})\|y\|$, and
  $|f_i(T_Ey) - f_i(y)| <  \frac{1}{n}\|y\|\cdot\|f_i\|$ for every $y
  \in E$ and $1 \le i \le n$. Extend $T_E$ (nonlinearly) to $Y$ by
  setting $S_E(y) = T_E(y)$ if $y \in E$
  and $S_E(y) = 0$ otherwise. Since $\|S_E(y)\| \le 2\|y\|$ and
  regarding $S_E(y)$ as an element in $M^{**}$ we can consider 
  \[(S_E)_{E \in I} \subset \Pi_{y \in Y} B_{M^{**}}(0, 2\|y\|).\]
  By Tychonoff's compactness theorem the net $(S_E)_{E \in I}$ has a
  convergent subnet $(S_{E_F})_{F \in J}$ in the product weak$^*$
  topology with limit $S$ say. This  means that for every finite
  number of points $(y_j)_{j \le K}   \subset Y$ we have $S_E y_j \to
  Sy_j$ with respect to the weak$^*$ topology on   $M^{**}$. Using
  this it is easy to see that the mapping $S: Y \to M^{**}$ is linear,
  of norm $1$, and the identity on $M$. Now, if we define  $\varphi:
  M^* \to Y^*$ by  \[\varphi m^*(y) = m^*(Sy),\] it is straightforward
  to check that $\varphi \in \HB(M,Y)$ with $\varphi(M^*) \supset W$.  

  Finally we check that condition (b) in Lemma
  \ref{lem:ai-equi} holds for $\varphi$. To
  this end let $H$ be a finite dimensional subspace of $Y$ and $G$ a
  finite dimensional subspace of $M^*$. Let $(g_l)_{l = 1}^m$ be
  a $\delta$-net for $S_G$ and choose $n$ so big that $S_{B_n}$ contains  
  a $\delta$-net $(h_p)_{p = 1}^q$ for $S_{H  \cap M}$ where
  $(h_p)_{p = 1}^r$, $r \ge q$, is a $\delta$-net for $S_H$. Then choose $H'
  \in J$ with $E_{H'} \supset \spann(H, B_n)$ such that the linear operator $T_{E_H'}:
  E_{H'} \to M$ satisfies $|\varphi g_l(h_p) - g_l(T_{E_H'}h_p)| < \delta$ for every $l \le
  m, p \le r$. Note that $T_{E_H'}|_{B_N} = I_{B_N}$ and
  $(1-\frac{1}{N})\|h\|\le\|T_{E_{H'}}h\|  \le (1 +
  \frac{1}{N})\|h\|$ for every $h \in E_{H'}$
  where $N > n$ is the unique number such that $E_{H'} \in I_N$. Now,
  for $h \in S_{H \cap M}$ we can find  $h_{p'} \in  (h_p)_{p=1}^q$
  such that $\|h - h_{p'}\| < \delta$. Thus we get  
  \begin{align*}
    \|T_{E_{H'}}h - h\| &\le \|T_{E_{H'}}h - T_{E_{H'}}h_{p'}\| +
                          \|T_{E_{H'}}h_{p'} - h_{p'}\| + \|h_{p'} - h\| \\ 
                        & \le (2 + 1/N)\delta.      
  \end{align*}
  
  For $h \in S_H$, $g \in S_G$ find $h_{p^{''}} \in (h_p)_{p=1}^r$ and
  $g_{l'} \in (g_l)_{l=1}^m$ with $\|h - h_{p^{''}}\| < \delta$ and
  $\|g - g_{l'}\| < \delta$. We get 
  \begin{align*}
    |\varphi g(h) - g(T_{E_{H'}}h)| 
   & \le |\varphi g(h) - \varphi
    g_{l'}(h)| + |\varphi g_{l'}(h) - \varphi g_{l'}(h_{p^{''}})| \\                                         
   & \hskip 2mm + |\varphi g_{l'}(h_{p^{''}}) -
     g_{l'}(T_{E_{H'}}h_{p^{''}})| + |g_{l'}(T_{E_{H'}}h_{p^{''}}) -
     g_{l'}(T_{E_{H'}}h)|\\                          
   & \hskip 2mm + |g_{l'}(T_{E_{H'}}h) - g(T_{E_{H'}}h)|\\ 
   & \le \|g - g_{l'}\| + \|h - h_{p^{''}}\| + \delta \\ & \hskip 2mm+
                                                           \|T_{E_{H'}}\|(\|h_{p^{''}}
                                                           - h\| + \|g_{l'} - g\|)\\
   & \le \delta(5 + 2/N).  
  \end{align*}
Now, as $\delta$ can be chosen arbitrary small, the operator
$T_{E_{H'}}$ restricted to $H$ will do the work. 
\end{proof}

\begin{rem}
  Note that Theorem \ref{thm:abr} can by transfinite induction, just as
  \cite[III.Lemma~4.3]{HWW}, be extended to a
  non-separable version similar to
  \cite[III.Lemma~4.4]{HWW} but with an almost isometric Hahn-Banach
  extension operator in place of a Hahn-Banach extension operator. The
  proof is like that of  \cite[III.Lemma~4.4]{HWW}, but uses Theorem
  \ref{thm:abr} instead of \cite[III.Lemma~4.3]{HWW}, and with a final
  part similar to the last part of the proof of Theorem \ref{thm:abr}.  
\end{rem}

As pointed out in Section \ref{sec1} examples of ideals which
are not ai-ideals are plentiful. Similarly examples of ai-ideals which are not
strict ideals are also plentiful. Indeed, take any non-separable space $Y$
for which $Y^*$ contains no proper 1-norming subspace (e.g. the
case for spaces being M-ideals in their biduals \cite{GoSa} or more
generally for strict u-ideals in their biduals \cite{GKS}). Then for every
separable subspace $X$ of $Y$ there exists a separable ai-ideal $Z$ in
$Y$ containing $X$. This ai-ideal cannot be strict. Moreover, this reasoning
actually shows that one cannot extend the main theorem replacing
``ai-ideal'' with ``strict ideal''.

\section{Characterizations in terms of subspaces}\label{sec3}
Let $Y$ be a Banach space with unit ball $B_Y$. By a \emph{slice} of $B_Y$ we
mean a set $S(y^*, \eps)=\{y \in B_Y: y^*(y)>1-\eps\}$ where $y^*$ is
in the unit sphere $S_{Y^*}$ of $Y^*$ and $\eps > 0$. A \emph{finite
  convex combination of slices} of $B_Y$
is a set of the form \[S=\sum_{i=1}^{n}\lambda_i
S(y^*_i,\eps_i)\]
where $\lambda_i \ge 0, \sum_{i=1}^{n}\lambda_i=1$, $y^*_i \in
S_{Y^*}$, and $\eps_i > 0$ for $i=1,2,\ldots,n$.

The relations between the following three successively stronger
properties were investigated in \cite{MR3098474}:

\begin{defn} \label{defn:diam2p}A Banach space $Y$ has the
  \begin{enumerate}
  \item [(i)] {\it local diameter 2 property (LD2P)} if every slice of $B_Y$ has
    diameter 2.
  \item [(ii)] {\it diameter 2 property (D2P)} if every
    non-empty relatively weakly open subset in $B_Y$ has diameter 2.
  \item [(iii)] {\it strong diameter 2 property (SD2P)} if every finite convex
    combination of slices of $B_Y$ has diameter 2.
  \end{enumerate}
\end{defn}
 
From \cite[Theorem~2.4]{BGLPRZ4} it is known that LD2P $\not\Rightarrow$
D2P  and from \cite[Theorem~1]{HL2} or \cite[Theorem~3.2]{ACBGLPR}
that D2P $\not\Rightarrow$ SD2P. 

Using Theorem \ref{thm:abr} we obtain in terms of subspaces
characterizations of the mentioned diameter 2 properties in the
following way.

\begin{prop}
  Let $Y$ be a Banach space. Then $Y$ has the $SD2P$ (resp. the $D2P$, $LD2P$)
  if and only if every separable ai-ideal in $Y$ does.  
\end{prop}

  \begin{proof}
  Sufficiency was established in \cite [Propositions~3.2, 3.3 and
  Corollary~3.4]{ALN2} where it was proved that every ai-ideal in $Y$
  has the SD2P (resp. the D2P, LD2P) whenever $Y$ has. 

  First let us prove the necessity for the SD2P. To this end let $\eps_k
  > 0$ for $k = 1, \ldots, n$ and $S = \sum_{k=1}^n \lambda_k S_k$ a
  finite convex combination of slices $S_k = \{y \in B_Y: \ys_k(y) >
  1- \eps_k\}$ of the unit ball of $Y$. By Theorem
  \ref{thm:abr} find a separable ai-ideal $X$
  in $Y$ such that $\spann(\ys_k)_{k = 1}^n \subset
  \varphi(\Xs)$ where $\varphi \in \HB_{ai}(X,Y)$. For $k =
  1, \ldots, n$ find $\xs_i \in S_{X^*}$ such that $\ys_k =
  \varphi(\xs_k)$. Define the slices $S_k' = \{x \in B_X: \xs_k(x) > 1- \eps_k\} =
  \{x \in B_X: \varphi \xs_k(x) > 1- \eps_k\}$ and note that $S_k'
  \subset S_k$. Now the convex combination of slices $S' =
  \sum_{k=1}^n \lambda_k S_k'$ has diameter
  2 by assumption. As $S' \subset S$, we get that $S$ has diameter 2
  as well.

  For the LD2P the result follows by taking $k=1$
  in the argument above. 

  For the D2P property let $V$ be a
  relatively weakly open subset in $B_Y$. Find $y_0 \in V$ and
  $y_i^\ast \in Y^\ast$ such that $V_\eps=\{y \in B_Y:
  |y_i^\ast(y-y_0)|<\eps,
  i=1,\cdots,n\} \subset V$. By Theorem
  \ref{thm:abr} find a separable ai-ideal $X$
  in $Y$ which contains $y_0$ and such that $\spann(\ys_k)_{k = 1}^n \subset
  \varphi(\Xs)$ where $\varphi \in \HB_{ai}(X,Y)$. Then a
  similar argument as above will finish the proof.
\end{proof}

Recall that a Banach space $Y$ has the \emph{Daugavet property} if 
for every rank one operator $T: Y \to Y$ the equation
$\|I+T\| =  1 + \|T\|$ holds where $I$ is the identity operator on $Y$.
One can show that the Daugavet property is equivalent to the
following statement (see \cite{MR1784413}): For every $\eps > 0$, every $y^*_0 \in
S_{Y^*}$, and every $y_0 \in S_Y$ there exists a point $y$ in the slice $S(y^*_0,\eps_0)$
such that $\|y + y_0\| \ge 2 - \eps$. 

Using Theorem \ref{thm:abr} we get a similar characterization
of the Daugavet property as for the diameter 2 properties.

\begin{prop}
  A Banach space $Y$ has the Daugavet property if and only if every
  separable ai-ideal in $Y$ does.
\end{prop}

\begin{proof}
  That the Daugavet property is inherited by ai-ideals is proved in
  \cite[Proposition~3.8]{ALN2}. Let $\eps > 0$ and choose positive
  $\delta < \eps$ such that $(1- \delta)^2> 1 -
  \eps$. Let $y^*_0 \in S_{Y^*}$ and $y_0 \in S_Y$. We must show that
  the slice $S(y^*_0, \eps) = \{y  \in B_Y: y^*_0(y) > 1 - \eps\}$,
  contains $y$ such that $\|y_0 + y\| > 2 - \eps$. To this
  end, choose $y_1 \in S(y^*_0, \delta)$ and find a separable ai-ideal
  $Z$ which contains $\spann\{y_0, y_1\}$. By construction the slice
  $S(y^*_0|_Z/\|y^*_0|_Z\|, \delta)$ of $B_Z$ is non-empty and by
  assumption there exists $y \in S(y^*_0|_Z/\|y^*_0|_Z\|, \delta)$ with
  $\|y_0 + y\| > 2 - \delta$. Since $y^*_0(y) =
  \|y^*_0|_Z\|y^*_0|_Z(y)/\|y^*_0|_Z\| > \|y^*_0|_Z\| (1 - \delta) > (1 -
  \delta)(1 - \delta) > 1- \eps$,
  we are done.  
\end{proof}

In \cite{ALL} the notion of an almost square Banach space was
introduced and studied. 

\begin{defn}
  A Banach space $Y$ is said to be \emph{almost square (ASQ)} if for every
  $\eps > 0$ and every finite set $(y_n)_{n=1}^N \subset S_Y$, there
  exists $y \in S_Y$ such that
  \[\|y_n - y\| \le 1+ \eps \text{ for every } 1 \le n \le N.\]
\end{defn}

The notion of an ASQ space is in some sense, but not quite, dual to the well
established notion of an octahedral space (see Definition
\ref{defn:octa} and the paragraph that follows) introduced by Godefroy in
\cite{G-octa}. We will discuss octahedral spaces briefly below. 

Among the examples of ASQ spaces we find
the much studied class of (non-reflexive) M-embedded spaces
\cite[Corollary~4.3]{ALL}, i.e. spaces $Y$ of the form $Y^{***} = Y^*
\oplus_1 Y^\perp$ (cf. \cite{HWW} for the theory
of such spaces). It is not hard to see from Theorem
\ref{thm:asq-findim} below that ASQ spaces contain copies of
$c_0$. ASQ spaces also possess the SD2P (see
\cite[Proposition~2.5]{ALL} and \cite[Theorem~3.3]{HLP}), but
not all spaces with the SD2P are ASQ. Take
e.g. $C[0,1]$ which is Lindenstrauss and thus has the SD2P
\cite[Proposition~4.6]{ALN2}. Moreover, if we let $y_1$ to be the constant
one function and $y_2 = -y_1$ one easily sees that $C[0,1]$ fails to be ASQ. 

The following characterization of ASQ spaces was obtained in
\cite[Theorem~3.6]{ALL}. 

\begin{thm}\label{thm:asq-findim}
  Let $Y$ be a Banach space.
  If $Y$ is ASQ then for every
  finite dimensional subspace $E \subset Y$
  and $\eps > 0$
  there exists $y \in S_Y$ such that
  \begin{equation}\label{eq:asq}
    (1-\eps)\max(\|x\|,|\lambda|)
    \le \|x + \lambda y\|
    \le
    (1+\eps)\max(\|x\|,|\lambda|)
  \end{equation}
  for all $\lambda \in \err$ and all $x \in E$.
\end{thm}

We will use this result and Theorem \ref{thm:abr} to
obtain characterizations of an ASQ space in terms of its subspaces.

\begin{thm}\label{thm:asq-char}
  Let $Y$ be a Banach space. Then the following statements are
  equivalent.
  \begin{enumerate}
    \item[(a)]$Y$ is ASQ.
    \item[(b)]Every separable ai-ideal in $Y$ is ASQ.
    \item[(c)]Every subspace $X$ of $Y$ for which $Y/X$ does not
      contain a copy of $c_0$ is ASQ. 
    \item[(d)]Every subspace of finite codimension in $Y$ is ASQ. 
  \end{enumerate}
\end{thm}

\begin{proof}
  (a) $\Rightarrow$ (b). This is proved in
  \cite[Lemma~4.5]{ALL}. 
  
  (b) $\Rightarrow$ (a). Let $\eps>0$ and $(y_n)_{n
    = 1}^N \subset S_Y$, and find by Theorem \ref{thm:abr} a separable
  ai-ideal $Z$ in $Y$ containing $(y_n)_{n=1}^N$. As $Z$ is  ASQ there
  exists $z \in S_Z$ such that $\|y_n - z\| \le 1 + \eps$ and we are done.   

  (c) $\Rightarrow$ (d) is trivial and (d) $\Rightarrow$ (a) is clear as every finite set of points is contained in a subspace of
  finite codimension in $Y$. 
  
  (a)  $\Rightarrow$ (c). Let $(y_n)_{n=1}^N \subset S_Y$,
  $E=\spann(y_n)_{n=1}^N$, and $\eps, \delta > 0$ with
  $\frac{1+\delta}{1-\delta} + \delta < 1 + \eps$. Choose a sequence
  $(\delta_k)_{k=0}^\infty$ of positive  reals such that $(\delta_k)
  \downarrow 0$ 
  and  
  \begin{equation}
    \label{eq:9}
    \Pi_{k=0}^\infty(1-\delta_k) > 1 - \delta \text{ and }
  \Pi_{k=0}^\infty(1+\delta_k) < 1 + \delta.
  \end{equation}
  Using Theorem \ref{thm:asq-findim} we can find a sequence $(z_k)
  \subset S_Y$ such that for every $y \in \spann(E \cup \{z_1, \ldots,
  z_k\})$ and every $\lambda \in \mathbb R$ we have
  \begin{equation}
    \label{eq:2}
    (1-\delta_k)\max(\|y\|,|\lambda|) \le
    \|y + \lambda z_{k+1}\| \le (1+\delta_k)\max(\|y\|,|\lambda|).
   \end{equation} 
  Now, if $z=\sum_{k=1}^K\lambda_kz_k$ we get from (\ref{eq:9}) and
  (\ref{eq:2}) that
  \begin{align}
    \label{align:9}
     (1-\delta)\max(\|y\|,|\lambda_k|_{k=1}^K) & < \|y + z\| \\& <
       (1+\delta)\max(\|y\|,|\lambda_k|_{k=1}^K)\nonumber 
  \end{align}
  for every $y \in E$. It is clear that
  the space $\overline{\spann}(z_k)_{k=1}^\infty$ is isomorphic to
  $c_0$. As $Y/X$ does not contain a copy of $c_0$, the quotient map
  $\pi: Y \to Y/X$ fails to be bounded below on
  $\overline{\spann}(z_k)_{k=1}^\infty$. From this
  it follows that there exists a linear combination
  $\sum_{k=1}^{K_1}\lambda_k z_k$ whose norm is 1 and with
  $\|\pi(\sum_{k=1}^{K_1}\lambda_k z_k)\| \le \delta/4$. Thus there is
  $f \in X$ with $\|f - \sum_{k=1}^{K_1}\lambda_k z_k\| \le
  \delta/2$. Putting $g=f/\|f\|$ we get $\|g -
  \sum_{k=1}^{K_1}\lambda_k z_k\| \le \delta$. Hence using
  (\ref{align:9}) we get for $n=1, \ldots, N$
  \begin{align*}
    \|y_n - g\|&\le \|y_n + \sum_{k=1}^{K_1}\lambda_k z_k\| +
    \|g - \sum_{k=1}^{K_1}\lambda_k z_k\|\\ & <
    (1+\delta)\max(\|y_n\|,|\lambda_k|_{k=1}^{K_1}) + \delta\\ & \le
                                                                 \frac{1+\delta}{1-\delta}
                                                                 +
                                                                 \delta
                                                                 < 1 + \eps, 
  \end{align*}
which is what we need.
\end{proof}

\begin{rem}
  From Theorem \ref{thm:abr} and Theorem \ref{thm:asq-char} (a)
  $\Rightarrow$ (b) we get that every separable subspace of an ASQ
  space $Y$ is contained in a separable subspace $Z$ in $Y$ which is
  both ASQ and ai-ideal in $Y$. This improves
  \cite[Proposition~6.5]{ALL} which says only that $Z$ can be taken to
  be ASQ. Theorem  \ref{thm:asq-char} (a)   $\Leftrightarrow$ (b) should also be
  compared with \cite[Lemma~4.1]{ALL}
\end{rem}

Let us end the paper with a result similar to Theorem
\ref{thm:asq-char} for octahedral spaces. 
   
\begin{defn}\label{defn:octa}
  A Banach space $Y$ is said to be \emph{octahedral} if for every
  $\eps > 0$ and every finite set $(y_n)_{n=1}^N \subset S_Y$, there
  exists $y \in S_Y$ such that
  \[\|y_n - y\| \ge 2- \eps \text{ for every } 1 \le n \le N.\]
\end{defn}

An easy consequence of the PLR is that spaces $Y$ of the form
$Y^{**} = Y \oplus_1 X$ where $X$ is a subspace of $Y^{**}$, i.e. the
L-embedded spaces (cf. \cite{HWW} for the theory of such spaces) are
octahedral. It is not so hard to see that octahedral spaces contain $\ell_1$
\cite{G-octa}. Quite recently octahedral spaces were studied in
\cite{BGLPRZ1} and \cite{HLP}. In \cite[Theorem~3.3]{HLP} it was
proved that a space has the SD2P if and only if its dual is
octahedral . Characterizations of the D2P and the LD2P in terms of
weaker forms of octahedrality in the dual were also established in
this paper (see \cite[Theorems~2.3 and 2.7]{HLP}). 

In \cite{MR2844454} it was proved that spaces with the almost Daugavet
property are octahedral and that in the separable case the two
properties are equivalent. Recall that a Banach space is said to be an \emph{almost
  Daugavet space} if there exists a 1-norming subspace $X$ of $Y^*$
such that for every $y^*_0 \in S_X$, every $y_0 \in S_Y$, and every $\eps > 0$
there exists a point $y$ in the slice $S(y^*_0,\eps)$ such that $\|y + y_0\| \ge 2 -
\eps$.

L{\"u}cking in \cite[Theorem~2.5]{MR2801618}
proved that separable almost Daugavet spaces satisfies
statement (c) in Theorem \ref{thm:octa-char} below. For non-separable
spaces it is as far as the author knows unknown whether octahedrality
implies almost Daugavet. We obtain, in terms of subspaces, the
following characterization of octahedral spaces.     

\begin{thm}\label{thm:octa-char}
  Let $Y$ be a Banach space. Then the following statements are
  equivalent.
  \begin{enumerate}
    \item[(a)]$Y$ is octahedral.
    \item[(b)]Every separable ai-ideal in $Y$ is octahedral.
    \item[(c)]Every subspace $X$ of $Y$ for which $Y/X$ does not
      contain a copy of $\ell_1$ is octahedral. 
    \item[(d)]Every subspace of finite codimension in $Y$ is octahedral. 
  \end{enumerate}
\end{thm}

The proof of this result follows along the same lines as the proof
Theorem \ref{thm:asq-char} using \cite[Proposition~2.4]{HLP}
instead of Theorem \ref{thm:asq-findim} and otherwise adjusting to the
$\ell_1$ setting. Therefore the proof will be omitted. Theorem
\ref{thm:octa-char} should be compared with \cite[Theorems~2.2 and
2.6]{BGLPRZ3}.

\section*{Acknowledgements}
The author would like to thank Professor Åsvald Lima for pointing at
the right literature to study when addressing the main problem of
the paper. 

\def\cprime{$'$} \def\cprime{$'$}
\providecommand{\bysame}{\leavevmode\hbox to3em{\hrulefill}\thinspace}
\providecommand{\MR}{\relax\ifhmode\unskip\space\fi MR }
\providecommand{\MRhref}[2]{%
  \href{http://www.ams.org/mathscinet-getitem?mr=#1}{#2}
}
\providecommand{\href}[2]{#2}

\end{document}